\documentclass[11pt]{article}
\usepackage[top=1in, bottom=1in, left=1in,right=1in]{geometry}
\usepackage{amsmath}
\usepackage{amssymb}
\usepackage{titlesec}
\usepackage{color}
\usepackage{amsthm}
\usepackage{mathrsfs}   
\usepackage{theoremref}
\usepackage{stix}

\usepackage[all]{xy}
\usepackage{tikz-cd}

    \newcommand\lmod[2]{
        \mathchoice
            {
                \text{\raise1ex\hbox{$#1$}\Big/\lower1ex\hbox{$#2$}}%
            }
            {
                #1\,/\,#2
            }
            {
                #1\,/\,#2
            }
            {
                #1\,/\,#2
            }
    }

\newtheorem{defn}{Definition}[section]
\newtheorem{thm}[defn]{Theorem}
\newtheorem{lem}[defn]{Lemma}  
\newtheorem{prop}[defn]{Proposition}
\newtheorem{cor}[defn]{Corollary}
\theoremstyle{definition}

\theoremstyle{definition}

\titleformat{\section}{\large\bfseries}{ }{12pt}{}

\title{Explicit Description of Centralizers for a Matrix}
\author{Tianhao Wang}
\date{\today}

\begin{document}
\maketitle
\begin{abstract}
    Let $k$ be a field and $A\in M_n(k)$ be an $n\times n$ matrix. We denote $C_{M_n(k)}(A) = \{B\in M_n(k) : BA = AB\}$ be its centralizers in $M_n(k)$. The dimension of the space of centralizer was already known by Frobenius. This paper will give the explicit $k$-basis for $C_{M_n(k)}(A)$ and also an algorithm (with polynomial complexity respect to multiplication  in the field $k$) to construct the explicit basis. Lastly, the result can be used to solve a weaker version of the Wild Problem.
\end{abstract}
\tableofcontents
\newpage

\section{Notations and Preliminary Results}
Here are some standard results from Algebra that I need to use. Most of them can be found in standard Algebra book.

Let $A\in M_n(k)$ be an $n\times n$ matrix, we denote $V_A$ to be the $k[x]$-module $k^n$ defined via $x\cdot v = Av$. Let $B\in M_m(k)$ be another matrix, we define \[
A\oplus B = \left[
        \begin{tabular}{c|c}
             A & 0  \\ \hline
             0 & B 
        \end{tabular}
    \right] \in M_{n+m}(k)
\]
and it is easy to see that $V_A\oplus V_B \simeq V_{A\oplus B}$.

Let $f(x) = x^n + c_{n-1}x^{n-1} + \cdots + c_0\in k[x]$ be a monic polynomial, we define the \textbf{Companion Matrix} of $f$, denoted as $C(f)$ via 
\[
C(f) = \begin{bmatrix} 0 & 0 & \cdots & 0 & -c_0 \\ 1 & 0 &\cdots & 0 & -c_1 \\ 0 & 1 & \cdots & 0 & -c_2 \\ \vdots & & \ddots &  &\vdots \\ 0 & 0 & \cdots & 1 & -c_{n-1} \end{bmatrix}
\]
We also have $V_{C(f)}\simeq k[x]/fk[x]$ as $k[x]$-module via $e_i\mapsto x^{i-1}$. We are going to use this $k[x]$-isomorphism explicitly later, so we give it a name as $T$.

\begin{lem}
Let $A,B\in M_n(k)$, if $V_A\simeq V_B$ as $k[x]$-module, then there is $P\in GL_n(k)$ such that $A = PBP^{-1}$.
\end{lem}

\begin{thm}(\textbf{Rational Canonical Form})

Given $A\in M_n(k)$ and let $V_A$ be the corresponding $k[x]$-module. Then there are unique monic polynomials $f_1,\ldots f_m\in k[x]$ such that $f_1|f_2|\cdots | f_m$ and $$V_A\simeq \lmod{k[x]}{f_1k[x]}\oplus \cdots \oplus \lmod{k[x]}{f_mk[x]}\simeq V_{C(f_1)}\oplus\cdots \oplus V_{C(f_m)}$$
The $f_1,\ldots, f_m$ are called \textbf{invariant factors} of the matrix $A$, and by Lemma 1.1, we know that $A$ is similar to $C(f_1)\oplus \cdots \oplus C(f_m)$ and we call the latter as the \textbf{Rational Canonical Form} of $A$.
\end{thm}

\begin{lem}
Suppose $A = PBP^{-1}$, then $C_{M_n(k)}(A) = PC_{M_n(k)}(B)P^{-1}$.
\end{lem}
\begin{proof}
$C\in C_{M_n(k)}(A)$ iff $CA = AC$ iff $CPBP^{-1} = PBP^{-1}C$ iff $(P^{-1}CP)B = B(P^{-1}CP)$ iff $P^{-1}CP\in C_{M_n(k)}(B)$ iff $C\in PC_{M_n(k)}(B)P^{-1}$
\end{proof}

We will first construct centralizer for $C(f_1)\oplus\cdots\oplus C(f_m)$ in the third section. Then, I will give an algorithm to find the all the invariant factors and transformation matrix $P$ in the fourth section. This will give full description of $C_{M_n(k)}(A)$.

\newpage
\section{Dimension of $C_{M_n(k)}(A)$ as $k$-vector space}
\begin{lem}
There is an bijection between 
$$C_{M_n(k)}(A) \longleftrightarrow \operatorname{End}_{k[x]}(V_A)$$
\end{lem}
\begin{proof}
Given a matrix $B\in C_{M_n(k)}(A)$, we define $\phi: V_A\mapsto V_A$ via $v\mapsto B\cdot v$. Then $\phi(x\cdot v) = \phi (Av) = B\cdot Av = AB\cdot v = x\cdot \phi(v)$. Hence $\phi\in \operatorname{End}_{k[x]}(V_A)$.

Conversely, given $\phi: V_A\mapsto V_A$ an $k[x]$-module homomorphism, it is in particular a $k$-linear map and hence we can find its canonical matrix representation $B\in M_n(k)$ such that $\phi(v) = Bv$. Then we have $\phi(x \cdot v) = x\cdot \phi(v)$ for all $v\in k^n$ which implies $BAv = ABv$ for all $v\in k^n$. Hence $BA=AB$ and $B\in C_{M_n(k)}(A)$.
\end{proof}

We will analyze the structure of $\operatorname{End}_{k[x]}(V_A)$ first. Here are some technical lemmas we need.

\begin{lem}
Let $R$ be a ring and $M_i, N_i$ be $R$-modules. Then
$$\operatorname{Hom}_R\left(\bigoplus_{i=1}^m M_i,\; \bigoplus_{j=1}^n N_j\right)\simeq \bigoplus_{i,j} \operatorname{Hom}_R(M_i, N_j) $$
\end{lem}
\begin{proof}
Here is the explicit isomorphism.
$$\Phi: \bigoplus_{i,j} \operatorname{Hom}_R(M_i, N_j) \mapsto \operatorname{Hom}_R\left(\bigoplus_{i=1}^m M_i,\; \bigoplus_{j=1}^n N_j\right)$$ via
$$\big(\Phi(\phi_{ij})\big)(m_1,\ldots, m_n) = \left(\sum_{i=1}^m \phi_{i1}(m_i)\;,\;\ldots\;,\; \sum_{i=1}^m \phi_{in}(m_i)\right)$$
Checking that this is an $R$-module isomorphism is easy. We are going to use this map explicitly later.
\end{proof}

Based on this lemma, to understand $\operatorname{End}_{k[x]}(V_A)$, it suffice to understand each \[\operatorname{Hom}_{k[x]}\left(\lmod{k[x]}{f_ik[x]}, \lmod{k[x]}{f_jk[x]}\right)\]

\begin{lem}
Let $D$ be a PID and $I=aD, J = bD$ be two ideals of $D$. Then we define $(I : J) = \{x\in D : xJ\subset I\}$ which is an ideal in $D$ containing $I$. Denote $dD = (I: J)$. We will have $d = \operatorname{lcm}(a,b)/b$

Then $$\operatorname{Hom}_D \left(\lmod{D}{J}, \lmod{D}{I}\right) = \{ 1\mapsto \lambda d+I : \lambda\in D\}$$
\end{lem}

\begin{proof}
Checking that $(I : J)$ is an ideal containing $I$ is trivial. Since $D$ is principle, we can find $d\in D$ such that $dD = (I : J)$. Since $(I : J)J \subset I$, we have $dDbD\subset aD$ and hence $a|db$, and clearly $b|db$ and hence $\operatorname{lcm}(a,b)/b \,\big|\,d$. It is easy to check that $(\operatorname{lcm}(a,b)/b) J \subset I$ since $\operatorname{lcm}(a,b)/b \cdot \lambda b = \lambda \operatorname{lcm}(a,b)\in aD$. Hence $\operatorname{lcm}(a,b)/b \in dD$ and we have $d = \operatorname{lcm}(a,b)/b$.

Note that $(D/J)$ has dimension $1$ as $D$-module and $1+J\in D/J$ is its generator. Hence $\operatorname{Hom}_D(D/J, D/I)$ depends uniquely on the image of $1+J$. 
Pick $\phi\in \operatorname{Hom}_D(D/J, D/I)$, for all $\lambda\in J$, we have $\phi(0) = \phi(\lambda\cdot (1+J)) = \lambda\cdot\phi(1+J) +I = 0$. Hence $J\phi(1+J)\in I$ and we have $\phi(1+J)\in (I : J)$ and it is easy to check that all $1\mapsto \lambda d + I$ defines an element in $\operatorname{Hom}_D(D/J, D/I)$. 
\end{proof}

\begin{cor}
Pick $f_1,f_2\in k[x]$ be non-zero polynomial such that $f_1|f_2$ and let $q = f_2/f_1$. Then we have
$$\operatorname{Hom}_{k[x]}\left(\lmod{k[x]}{f_1k[x]}, \lmod{k[x]}{f_2k[x]}\right) = \{1\mapsto \lambda q+f_2k[x] : \lambda\in k[x]\}$$

$$\operatorname{Hom}_{k[x]}\left(\lmod{k[x]}{f_2k[x]}, \lmod{k[x]}{f_1k[x]}\right) = \{1\mapsto \lambda + f_1k[x] : \lambda \in k[x]\}$$

Also, both of them as $k$-vector space has dimension $\deg f_1$.
\end{cor}

\begin{proof}
The first part is a direct result from the above lemma. For the second part, the $k$-dimension for the first space is $\deg f_1$ since if $\deg \lambda > \deg f_1$, then $\lambda = f_1a+b$ where $a,b\in k[x]$ and $\deg b < \deg f_1$. Then $\lambda q +f_2k[x]= (f_1a+b)\frac{f_2}{f_1} + f_2k[x] = f_2a + bq +f_2k[x] = bq+f_2k[x]$ where $\deg b < \deg f_1$.

The $k$-dimension for the second space is obvious.
\end{proof}

Now, we have a nice formula for the $k$-dimension of the centralizer of a square matrix $A$. This result was also proved by Frobenius.

\begin{prop}
Let $A\in M_n(k)$ and $f_1|\cdots | f_m$ be its invariant factors, then 
$$\dim_k C_{M_n(k)}(A) = \sum_{i=0}^{m-1} (2i+1)\deg f_{m-i} = \deg f_m + 3\deg f_{m-1} + 5\deg f_{m-2}+\cdots$$
\end{prop}

\begin{proof}
By Lemma 2.1, we have $\dim_k C_{M_n(k)}(A) = \dim_k \operatorname{End}_{k[x]}(V_A)$. Note that 
$$V_A\simeq \lmod{k[x]}{f_1k[x]}\oplus \cdots \oplus \lmod{k[x]}{f_mk[x]}$$
Using Lemma 2.2 and Corollary 2.4, we have
\begin{align*}
    \dim_k C_{M_n(k)}(A) &= \dim_k \left (\bigoplus_{i,j} \operatorname{Hom}_{k[x]}\left(\lmod{k[x]}{f_ik[x]}, \lmod{k[x]}{f_jk[x]}\right)\right) \\
    &= \sum_{i,j} \dim_k \operatorname{Hom}_{k[x]}\left(\lmod{k[x]}{f_ik[x]}, \lmod{k[x]}{f_jk[x]}\right)\\
    &= \sum_{i,j} \min({\deg f_i, \deg f_j})
\end{align*}
For invariant factor $f_{m-k}$, we have $k$ number of $f_j$ having $\deg$ greater than $\deg f_{m-k}$ which contributes $2k \deg f_{m-k}$ and $i=j=m-k$ contributes another $\deg f_{m-k}$. Then the formula follows.
\end{proof}

\begin{cor}
Let $A\in M_n(k)$. If $A$ has only $1$ invariant factor (When the characteristic polynomial $f_A$ equals the minimal polynomial $f_m$), then 
$$C_{M_n(k)}(A) = \{\lambda(A) : \lambda(x)\in k[x]\}$$
\end{cor}
\begin{proof}
$A$ clearly commutes with $\lambda(A)$ for all $\lambda(x)\in k[x]$. It is easy to check that the latter is a $k$-vector space with dimension $\deg f_m$. By previous proposition, we have $\dim_k C_{M_n(k)}(A) = f_m$ and hence they must be equal.
\end{proof}

\newpage
\section{Explicit Description for Centralizers of $C(f_1)\oplus \cdots \oplus C(f_m)$}

Suppose $f_1|\cdots|f_m\in k[x]$ be monic and non-zero polynomials and let $A = C(f_1)\oplus\cdots\oplus C(f_m)$. Denote $n_i = \deg f_i$ and $T_i: V_{C(f_i)}\mapsto k[x]/f_ik[x]$ be the canonical $k[x]$-module isomorphism via $e_j\mapsto x^{j-1}+f_ik[x]$. We define the \textbf{generating polynomial} $q_{ij}(x)\in k[x]$ via 

\[
q_{ij}(x) = \begin{cases}
1 \;\;\;\;\;\;\;\;\;\;\text{ for }i\leq j \\
f_i/f_j \;\;\;\;\text{ for } i>j
\end{cases}
\]

With a little bit ambiguity, we use to same notation to denote the \textbf{generating vector} $q_{ij} = {T_i}^{-1}(q_{ij}(x))\in k^{n_i}$

More explicitly, when $i>j$, suppose $q_{ij}(x) = f_i/f_j =  c_0+c_1x+\cdots x^{n_i-n_j}\in k[x]$, then we will have $q_{ij} = (c_0, c_1,\ldots, c_{n_i-n_j-1},0,\ldots,0)\in k^{n_i}$. When $i\leq j$, we will have $q_{ij} = e_1\in k^{n_i}$.

Then we define the \textbf{generating matrix} $Q_{ij}\in M_{n_i,n_j}(k)$ where 
$$Q_{ij} = 
\left[\begin{array}{c|c|c|c}
q_{ij} & C(f_i)q_{ij} &C(f_i)^2q_{ij} & \cdots
\end{array}\right]
$$

In particular, $Q_{ii} = I_{n_i}$.

\begin{thm} (\textbf{Centralizer for Matrix in Rational Canonical Form}) \par
Let $A = C(f_1)\oplus\cdots\oplus C(f_m)$, then all elements in $C_{M_n(k)}(A)$ has the following decomposition as block matrix 
\[
    \left[
    \begin{array}{c|c|c|c}
        C_{11} & C_{12} & \cdots & C_{1m} \\ \hline
        C_{21} & C_{22} & \cdots & C_{2m} \\ \hline
        \vdots & & \ddots & \vdots \\ \hline
        C_{m1} & C_{m2} & \cdots & C_{mm} 
    \end{array}
    \right]
\]

Where $C_{ij} \in M_{n_i, n_j}(k)$ and $C_{ij}$ depends on choice of $\phi_{ij}\in \operatorname{Hom}_{k[x]}(k[x]/f_j, k[x]/f_i)$. 

More explicitly, we have $$C_{ij} \in \{ \lambda(C(f_i))Q_{ij}: \lambda(x)\in k[x]\} = \operatorname{span}_k\{I, C(f_i),\ldots, C(f_i)^{n_i-1}\}Q_{ij}$$

Note that the right hand side is not $k$-linear independent when $i > j$. To make it linear independent, we pick $\operatorname{span}_k\{I,C(f_i),\ldots, C(f_i)^{n_j-1}\}Q_{ij}$ instead when $i>j$.

The explicit basis for $C_{M_n(k)}(A)$ is given by picking $C_{ij}$ from above and other blocks being $0$.

\end{thm}
\begin{proof}
To understand $C_{M_n(k)}(A)$, it suffice to understand $\operatorname{End}_{k[x]}(V_A)$ in an explicit manner and represent each element as a matrix. 

We proved so far that 
$$\operatorname{End}_{k[x]}V_A \simeq \bigoplus_1^m \operatorname{Hom}_{k[x]}\left(\lmod{k[x]}{f_ik[x]}, \lmod{k[x]}{f_jk[x]}\right)$$
Picking $(\phi_{ij})$ in the RHS, then this is mapped to $\phi\in \operatorname{End}_{k[x]}(V_A)$ via
$$\phi(a_1,\ldots, a_m) = \left(\sum_1^m \phi_{i1}(a_i),\ldots, \sum_1^m \phi_{im}(a_i)\right)$$

Consider the following commuting diagram 

\[ \begin{tikzcd}
\bigoplus_1^m k[x]/f_ik[x] \arrow{r}{\phi} & \bigoplus_1^m k[x]/f_ik[x] \arrow{d}{T^{-1} = ({T_1}^{-1},\ldots, {T_m}^{-1})} \\
\bigoplus_1^m V_{C(f_i)} \arrow{u}{T=(T_1,\ldots, T_m)} \arrow{r}{\Phi}& \bigoplus_1^m V_{C(f_i)}\\
\end{tikzcd}
\]

Pick $(v_1, \ldots, v_m)\in \bigoplus_1^m V_{C(f_i)}$, then we have 
\begin{align*}
    \Phi(v_1,\ldots,v_m) = T^{-1}\phi T (v_1,\ldots,v_m) &= T^{-1}\phi (T_1(v_1),\ldots, T_m(v_m)) \\ 
    &= T^{-1}\left(\sum_1^m \phi_{i1}(T_i(v_i)), \ldots, \sum_1^m \phi_{im}T_i(v_i)\right) \\
    &=  \left(\sum_1^m {T_1}^{-1}\phi_{i1}(T_i(v_i)), \ldots, \sum_1^m {T_m}^{-1}\phi_{im}T_i(v_i)\right)
\end{align*}

In particular, $\Phi(v_j) = \left( {T_1}^{-1}\phi_{i1}(T_j(v_j)), \ldots, {T_m}^{-1}\phi_{im}T_j(v_j)\right)$ (with a little bit ambiguity of internal and external direct sum, the $v_i$ means $(0,0,\ldots,v_i,\ldots,0)\in \bigoplus_1^m V_{C(f_i)}$). Hence represented as a matrix, we have the desired block decomposition with the block $C_{ij}$ depending on $\phi_{ij} \in \operatorname{Hom}_{k[x]} (k[x]/f_j, k[x]/f_i)$, and $C_{ij}$ is the matrix representation of $k[x]$-module homomorphism 
$$\Phi_{ij} : V_{C(f_j)} \mapsto V_{C(f_i)}\text{ via }  v_j\mapsto  T_i^{-1}(\phi_{ij}T_j(v_j))$$

To make it easier to understand $\Phi_{ij}$, I draw the commuting diagram here: 

\[ \begin{tikzcd}
 k[x]/f_jk[x] \arrow{r}{\phi} & k[x]/f_ik[x] \arrow{d}{{T_i}^{-1}} \\
 V_{C(f_j)} \arrow{u}{T_j} \arrow{r}{\Phi_{ij}}& V_{C(f_i)}\\
\end{tikzcd}
\]
By Corollary 2.4, we have 
\begin{align*}
\operatorname{Hom}_{k[x]}\left(\lmod{k[x]}{f_jk[x]}, \lmod{k[x]}{f_ik[x]}\right) &= \{1\mapsto \lambda\cdot q_{ij}(x) +f_jk[x] : \lambda\in k[x]\} \\
&= k[x](1\mapsto q_{ij}(x)+f_jk[x])
\end{align*}
and $q_{ij}(x) = 1$ when $i\leq j$ and $q_{ij}(x) = f_i/f_j$ when $i>j$. This matches our definition of the \textbf{generating polynomial}.

Note that all the $\phi_{ij}$ can be written as $1\mapsto \lambda(x)q_{ij}(x)$ for some $\lambda(x) \in k[x]$. Using the $k[x]$ module structure of $V_{C(f_i)}, V_{C(f_j)}$ and the property that $\Phi_{ij}$ is $k[x]$-module homomorphism, we have
$$\Phi_{ij}(v_j) = T_i^{-1}(\lambda(x)q_{ij}(x)T_j(v_j)) = \lambda(x)\cdot (T_i^{-1}(q_{ij}(x)T_j(v_j)) = \lambda(C(f_i))T_i^{-1}(q_{ij}(x)T_j(v_j))$$

We denote the matrix representation of $v_j \mapsto {T_i}^{-1}(q_{ij}(x)T_j(v_j))$ as $Q_{ij}$ and we have 
$$C_{ij} \in \{ \lambda(C(f_i))Q_{ij}: \lambda(x)\in k[x]\}$$

Now, it suffice to show that the matrix representation of the map $v_j \mapsto {T_i}^{-1}(q_{ij}(x)T_j(v_j))$ is the \textbf{generating matrix} that we defined previously.

The first column of $Q_{ij}$ is the image of $e_1$ which is sent to $T_i^{-1}(q_{ij}(x)\cdot 1) = T_i^{-1}(q_{ij}(x)) = q_{ij}$. Note that $T_j(e_k) = x^{k-1}\cdot T_j(e_1)$ and hence $$T_i^{-1}(q_{ij}(x)T_j(e_k)) = T_i^{-1}(q_{ij}(x)x^{k-1}T_j(e_1)) = x^{k-1}\cdot (T_i^{-1}(q_{ij}(x)T_j(e_1))) = C(f_i)^{k-1}q_{ij}$$

Hence this $Q_{ij}$ is exactly the \textbf{generating matrix} that we defined earlier.
\end{proof}

Next, I will give an algorithm for finding the rational canonical form of a matrix and also the transformation matrix.

\newpage
\section{Smith Normal Form and Invariant Factors}

We will first show the connection between the Smith Normal Form of $xI-A\in M_n(k[x])$ and invariant factors of $A\in M_n(k)$. Then I will give an algorithm with polynomial complexity for finding the Smith Normal Form.

\begin{lem}
Let $D$ be an PID, and $A\in M_{n,m}(D)$ be an $n\times m$ matrix. Denote $N_A(D) = \{v\in D^m: Av = 0\}$ and $R_A(D) = \{Av \in D^n: v\in R^m\}$. Let $F$ be the field of fraction of $D$, and define $N_A(F), R_A(F)$ similarly. Let $r = \dim_F R_A(F)$, then we have the following results: 
\begin{enumerate}
    \item $\lmod{D^m}{N_A(D)}$ is a free $D$-module and there is $x_1,\ldots, x_m\in D^m$ basis for $D^m$ such that $N_A(D) = Dx_{r+1}\oplus \cdots \oplus Dx_m$.
    \item There is $y_1,\ldots, y_n\in D^n$ $D$-basis for $D^n$ and $d_1|\cdots |d_n \in D$ such that $R_A(D) = Dd_1y_1\oplus\cdots\oplus Dd_ny_n$. 
\end{enumerate}
\end{lem}
\begin{prop}(\textbf{Smith Normal Form}) \par
With the same setup as the above lemma, we have
 $$A = [y_1,\ldots, y_n]\operatorname{diag}(d_1,\ldots, d_r, 0,\ldots, 0)[x_1,\ldots, x_m]^{-1}$$ where $[y_1,\ldots, y_n]\in GL_n(D)$, $[x_1,\ldots, x_m]\in GL_m(D)$ and $\operatorname{diag}(d_1,\ldots, d_r, 0,\ldots, 0)$ is the \textbf{Smith Normal Form} of $A$.
\end{prop}

\begin{proof}
Pick $v+N_A(D)\in D^m/N_A(D)$ and suppose there is $d\in D-\{0\}$ such that $d(v+N_A(D)) = 0$. Then we have $dv\in N_A(D)$ and hence $D(dv) = 0$. Since $d\neq 0$, we have $Dv = 0$ and hence $v\in N_A(D)$, $Tor(D^m/N_A(D)) = 0$. Since $D$ is a PID, $D^m/N_A(D)$ is free.

Since $N_A(D)$ is a submodule of a free module $D^m$, we can find $x_1,\ldots, x_m\in D^m$ and $a_1|\cdots | a_k \in D$ such that $D^m = \bigoplus_1^m Dx_i$ and $N_A(D) = \bigoplus_1^k Da_ix_i$. 
Then 
$$\lmod{D^m}{N_A(D)} \simeq \left(\bigoplus_1^k \lmod{D}{a_iD}\right) \oplus D^{m-k}$$
Since $D^m/N_A(D)$ is torsion free, we must have $a_i\in D^\times$. By changing $x_i$ to $a_ix_i$, we may assumes that $a_i = 1$, and by swapping the order of $x_i$'s we can have $N_A(D) = Dx_{m-k+1}\oplus\cdots \oplus Dx_m$.

Then note that $\operatorname{rank} (R_A(D)) = \dim_F R_A(F) = r$. Since $R_A(D)$ is a submodule of $D^n$ which is free, we can find $y_1,\ldots, y_n$ basis for $D^n$ and $d_1|\ldots| d_r\in D$ such that $R_A(D) = \bigoplus_1^r Dd_iy_i$. 

Consider the surjective $D$-module homomorphism $D^m\mapsto R_A(D)$ via $v\mapsto Av$ which have kernel $N_A(D)$. We have $D/N_A(D)\simeq R_A(D)$ via $v+N_A(D)\mapsto Av$. Hence $Dx_1\oplus \cdots\oplus Dx_{m-k} \simeq Dd_1y_1\oplus\cdots\oplus Dd_ry_r$. Hence $m-k = r$ and we have $N_A(D) \simeq \bigoplus_{r+1}^m Dx_i$. Moreover, since $d_1y_1,\ldots, d_ny_n$ is basis for $R_A(D)$. Then using the map $D^m/N_A(D) \simeq R_A(D)$ via $v+N_A(D)\mapsto Av$, we can find $x_1'\ldots, x_r'$ which are basis for $Dx_1\oplus \cdots \oplus Dx_r$ and $Ax_i' = d_iy_i$.

By changing $x_1,\ldots, x_r$ to $x_1',\ldots, x_r'$, we get the desired results.
\end{proof}

\begin{thm}(\textbf{Algorithm for Invariant Factors and Transformation Matrix}) \par
Let $k$ be a field, and $A\in M_n(k)$. Suppose 
\[
xI-A = \gamma_1(\operatorname{diag}(f_1,\ldots, f_n))\gamma_2
\]
is the Smith Normal Form of $xI-A\in M_n(k[x])$. Assume that $\gamma_1 = [y_1,\ldots, y_n]$. Let $m$ be the largest integer such that $\deg f_{m-1} = 0$. Then $f_m,\ldots, f_n$ are invariant factors of $A$, and the transformation matrix $P$ such that $A = P(C(f_m)\oplus\cdots\oplus C(f_n))P^{-1}$ is given by 
\[
    P = \left[ 
    \begin{array}{c|c|c|c|c|c|c|c}
         \phi(y_m) & A\phi(y_m) & \cdots & A^{\lambda_m-1}\phi(y_m) & \phi(y_{m+1}) &\cdots &A^{\lambda_{m+1}-1}\phi(y_{m+1}) & \cdots
    \end{array}
    \right]
\]
where $\lambda_m$ is the degree of the invariant factor $f_m$, and $\phi: k[x]^n\mapsto V_A$ is defined by $\sum_0^\infty x^iv_i \mapsto \sum_0^\infty A^iv_i$
\end{thm}

\begin{proof}
Let $r$ be the number of non-zero entries on the diagonal as previous. We first show that $r=n$ and hence the Smith Normal Form of $xI-A$ has no zero entries on the diagonal.

Note that $\det (xI-A)\neq 0$ since if not, we can find $v_f\neq 0 \in k[x]^n$ such that $Av_f = xv_f$. However, the deg of the polynomial on both sides does not match unless $v_f = 0$. Hence $\det (xI-A)\neq 0$ and we have $\det \operatorname{diag}(f_1,\ldots, f_r,0, \ldots ,0)\neq 0$. Hence we must have $r=n$.

Consider the map $$\phi: k[x]^n\mapsto V_A$$
via
$$\sum_0^\infty x^iv_i \mapsto \sum_0^\infty A^iv_i\;\;\;\;\;\;\;v_i\in k^n \text{ and only finite \# of } v_i\neq 0 $$
It is easy to check that $\phi$ is a surjective $k[x]$-module homomorphism. (Surjectivity is from $\phi(v) = v$) We claim that $$\ker \phi = R_{xI-A}(k[x]) = k[x]f_1y_1\oplus\cdots\oplus k[x]f_ny_n$$
(The second equality is from the proof for Smith Normal Form, where $y_1,\ldots, y_n$ is a basis for $k[x]^n$)

Pick $v = (xI-A)\sum_0^\infty x^iv_i = \sum_0^\infty x^{i+1}v_i - \sum_0^\infty x^i(Av_i)\in R_{xI-A}(k[x])$, then $$\phi(v) = \sum_0^\infty A^{i+1}v_i - \sum_0^\infty A^i(Av_i) = 0$$ Hence $R_{xI-A}(k[x])\subset \ker\phi$. 

Conversely, suppose $\sum_0^\infty x^iv_i\in\ker\phi$. Then we have $\sum_0^\infty A^iv_i = 0$. WLOG, assume that $v_i = 0$ for $i> k$. Then we have 
\begin{equation}
    v_0 + Av_1+\cdots A^kv_k = 0
\end{equation}
We will construct $v_0',\ldots, v_k'$ inductively such that $v_i = v_{i-1}'-Av_i'$ with $v_{-1}'=0$: \\
Let $v_0' = v_1+\ldots + A^{k-1}v_k$, then note that $-Av_0' = v_0$ by (1).\\
Having defined $v_i'$, we consider $A^{i+1}(v_{i+1} - v_i') = A^{i+1}v_{i+1} - A^{i+1}v_i' = A^{i+1}v_{i+1} - A^{i}(v_{i-1}'-v_i)$. Inductively, we have $A^{i+1}(v_{i+1}-v_i') = A^{i+1}v_{i+1} + A_iv_i+\cdots v_0$. Using equation 1), we have  
$$A^{i+1}(v_{i+1}-v_i') = -(A^{i+2}v_{i+2}+\cdots + A^kv_k)$$
and hence 
$$v_{i+1} - v_i' = -A(v_{i+2}+\cdots + A^{k-i-1}v_k)$$
Then $v_{i+1}' = v_{i+2}+\cdots + A^{k-i-1}v_k$ satisfies the desired property.

Then $\sum_0^\infty x^iv_i = \sum_0^\infty x^i(v_{i-1}'-Av_i') = \sum_0^\infty x^{i+1}v_i'-\sum_0^\infty Ax^iv_i' = \sum_0^\infty (xI-A)x^iv_i'\in R_{xI-A}(k[x])$.

Hence we have $$R_{xI-A}(k[x]) = \ker\phi = k[x]f_1y_1\oplus\cdots\oplus k[x]f_ny_n$$
Then using the surjective $k[x]$-module homomorphism $\phi$, we have 
$$\lmod{k[x]^n}{R_{xI-A}(k[x])} = \bigoplus_1^n k[x]/f_i(x)k[x] = \bigoplus_m^n k[x]/f_i(x)k[x] \simeq V_A$$

By Uniqueness of Rational canonical Form, we know that $f_m,\ldots, f_n$ are invariant factor for $A$.

To find the transformation matrix $P\in GL_n(k)$ such that $A = P(C(f_m)\oplus\cdots\oplus C(f_n))P^{-1}$, it suffice to understand the isomorphism
$$V_A\simeq \bigoplus_m^n k[x]/f_i(x)k[x]\simeq \bigoplus_m^n V_{C(f_i)}$$ explicitly.

Let $\lambda_i = \deg f_i$, We consider following diagram, where $T$ is the standard map as in previous section: $T: V_{C(f)}\mapsto k[x]/fk[x]$ via $e_i\mapsto x^{i-1}$:
\[ \begin{tikzcd}
 \bigoplus_m^nk[x]/f_jk[x] \arrow{r}{i} & \bigoplus_1^n k[x]y_i/f_iy_ik[x] \arrow{d}{\phi} \\
 \bigoplus_m^nV_{C(f_j)} \arrow{u}{T} \arrow{r}{}& V_A\\
\end{tikzcd}
\]
To construct $P$, we need to know image of $e_i$. Consider $e_1$, $e_1$ will be sent to $T(e_1) = (1,0,\ldots,0)$, $i(T(e_1)) = y_m$, and $\phi(i(T(e_1))) = \phi(y_m)$;

Similarly, if $\lambda_{m}\geq 2$. then using $T(e_2) = xT(e_1)$, we have $\phi(i(T(e_2))) = \phi(xy_m) = x\phi(y_m) = A\phi(y_m)$.

Hence 
\[
    P = \left[ 
    \begin{array}{c|c|c|c|c|c|c|c}
         \phi(y_m) & A\phi(y_m) & \cdots & A^{\lambda_m-1}\phi(y_m) & \phi(y_{m+1}) &\cdots &A^{\lambda_{m+1}-1}\phi(y_{m+1}) & \cdots
    \end{array}
    \right]
\]
\end{proof}
In summary, given $A\in M_n(k)$, if we want to find $C_{M_n(k)}(A)$, we first find the Smith Normal Form of $xI-A\in M_n(k[x])$. Using what has been described above, we can find all invariant factors $f_1,\ldots, f_m$ of $A$ and the transformation matrix $P$ such that $A = PC(f_1)\oplus\cdots\oplus C(f_m)P^{-1}$. Then $C_{M_n(k)}(A) = PC_{M_n(k)}(C(f_1)\oplus\cdots \oplus C(f_m))P^{-1}$ where $C_{M_n(k)}(C(f_1)\oplus\cdots \oplus C(f_m))$ is described in the previous section.

Finding the Smith Normal Form of a matrix and the transformation matrix has polynomial complexity (See \cite{2} for a detailed discussion of the complexity). The $k$-dimension of the centralizer is at most $n^2$, and constructing the each basis element involves at most one step of polynomial division and $n$ steps of matrix multiplying a vector. Hence the whole algorithm for producing the $k$-basis for $C_{M_n(k)}(A)$ has polynomial complexity.

A sample implementation in C++ which supports $k$ to be $\mathbb{Z}/p\mathbb{Z}$ can be found at \newline
https://github.com/TianhaoW/CentralizerOfMatrix. 

\newpage 
\section{Examples}
Suppose 

\[
A = C(x^2+1)\oplus C(x^3+x^2+x+1) = 
\begin{bmatrix}
0 & 1 & 0 & 0 & 0 \\ 1 & 0 & 0 & 0 & 0 \\ 0 & 0 & 0 & 0 & 1\\ 0 & 0 & 1 & 0 & 1 \\ 0 & 0 & 0 & 1 & 1
\end{bmatrix}
\in M_5(\mathbb{F}_2)
\]

We have $q_{21}(x) = x+1$, $q_{12} = (1,0)$ and $q_{21} = (1, 1, 0)$. Then we calculate the generating matrix, we have 
$Q_{12} = \begin{bmatrix} 1 & 0 & 1 \\ 0 & 1 & 0 \end{bmatrix}$, $Q_{21} = \begin{bmatrix} 1 & 0  \\ 1 & 1 \\ 0 & 1 \end{bmatrix}$ and $Q_{11} = I_2$, $Q_{22} = I_3$. Then we have

\[
C_{M_n(\mathbb{F}_2)}(A) = \left\{\left[
\begin{array}{c|c}
     C_{11} & C_{12} \\\hline
     C_{21} & C_{22} 
\end{array} \right]\right\}
\]

with $C_{11} \in \operatorname{span}_{\mathbb{F}_2}\left\{I_2, \begin{bmatrix}
0 & 1 \\ 1 & 0
\end{bmatrix}\right\} $
, $C_{22} \in \operatorname{span}_{\mathbb{F}_2}\left\{I_3, \begin{bmatrix}
0 & 0 & 1 \\ 1 & 0 & 1 \\ 0 & 1 & 1
\end{bmatrix}, \begin{bmatrix}
0 & 0 & 1 \\ 1 & 0 & 1 \\ 0 & 1 & 1
\end{bmatrix}^2\right\}$

$C_{12} \in \operatorname{span}_{\mathbb{F}_2}\left\{Q_{12}, \begin{bmatrix}
0 & 1 \\ 1 & 0
\end{bmatrix}Q_{12}\right\} $
, $C_{21} \in \operatorname{span}_{\mathbb{F}_2}\left\{Q_{21}, \begin{bmatrix}
0 & 0 & 1 \\ 1 & 0 & 1 \\ 0 & 1 & 1
\end{bmatrix}Q_{21}
\right\}$

$\empty$

Here is another example by running the sample code with $k = \mathbb{Z}/5{\mathbb{Z}}$ and $A = \begin{bmatrix}0 & 1 & 3 \\ 3 & 2 & 4 \\ 0 & 0 & 4 \end{bmatrix}$. The output for the explicit basis for the space of centralizer is 
\[
\left\{
\begin{bmatrix}
1 & 3 & 0 \\ 0 & 0 & 0 \\ 0 & 0 & 0
\end{bmatrix},
\begin{bmatrix}
0 & 0 & 2 \\ 0 & 0 & 0 \\ 0 & 0 & 0
\end{bmatrix},
\begin{bmatrix}
4 & 2 & 0 \\ 2 & 1 & 0 \\ 3 & 4 & 0
\end{bmatrix},
\begin{bmatrix}
0 & 2 & 0 \\ 0 & 1 & 0 \\ 0 &0 & 1
\end{bmatrix},
\begin{bmatrix}
0 & 1 & 3 \\ 0 & 3 & 4 \\ 0 &0 & 4
\end{bmatrix}
\right\}
\]

\newpage
\section{The Wild Problem}
One version of the wild problem can be described as the following: Let $k$ be a field, and $A,B,A',B'\in M_n(k)$, determine if $(A,B)$, $(A',B')$ are simultaneously similar or not. In other words, determine if there is a matrix $P\in GL_n(k)$ such that $PAP^{-1} = A'$ and $PBP^{-1} = B'$.

If we relax the condition of $PAP^{-1} = A'$ and $PBP^{-1} = B'$ to $PA = A'P$, $PB = B'P$ (drop the invertibility of $P$), the complete description of $C_{M_n(k)}(A)$ will give an algorithm with polynomial complexity to determine this.

Starting with the pair $(A,B), (A', B')$ and suppose $A\sim A'$, $B\sim B'$, we first find the Rational Canonical Form of $A,A',B, B'$. Denote $R_A$ as the Rational Canonical Form of $A,A'$ ($A\sim A'$, they have the same Rational Canonical Form), and $R_B$ as the Rational Canonical Form of $B,B'$. 

Suppose $PAP^{-1} = A'$ and $QBQ^{-1} = B'$, (an example of $P,Q$ can be found from the transformation matrix to $R_A$ and $R_B$). If $P_1$ also satisfies $P_1A = A'P_1$, then we have $P_1A = PAP^{-1}P_1$ which implies $(P^{-1}P_1)A = A(P^{-1}P_1)$. Hence $P^{-1}P_1\in C_{M_n(k)}(A)$, and we have $P_1\in PC_{M_n(k)}(A)$. It can be verified easily that the converse also holds.

In other words, all the matrix $P_1$ such that $P_1A = A'P_1$ are given by $P_1\in PC_{M_n(k)}(A)$. To determine if there is $U$ such that $UA = A'U$ and $UB = B'U$, it suffice to check if $PC_{M_n(k)}(A)\cap QC_{M_n(k)}(B) = \emptyset$.

As we showed earlier, $C_{M_n(k)}(A)$ is a $k$-vector space and we know its basis explicitly. Also, \newline 
$PC_{M_n(k)}(A), QC_{M_n(k)}(B)$ are also $k$-vector spaces and their basis are given by left multiplication of $P,Q$ to the basis of $C_{M_n(k)}(A), C_{M_n(k)}(B)$.

Now, the question is reduced to determine if the intersection of two vector spaces with known basis are empty or not. This can be solved easily by doing row reductions.

The set of $U$ such that $UA = A'U$ and $UB = B'U$ is a $k$-vector space. If we want to solve the wild problem, we need to determine if this $k$-vector space of matrices has invertible elements or not, for which I do not have a solution.

\end{document}